\documentclass[12pt]{article}
\usepackage[dvips]{graphics}
\usepackage{bm,enumerate,amscd,amssymb,amsthm,amsmath,subfigure,graphicx,youngtab,young}
\usepackage{dsfont,pgfkeys,pgfopts}
\usepackage{pifont}
\usepackage{amsfonts}
\usepackage{mathrsfs}

\def\min{\mbox{\rm min}}

\def\sh{\mbox{\rm sh}}

\def\And{\mbox{\rm ~and~}}

\def\If{\mbox{\rm ~if~}}

\def\type{\mbox{\rm type}}

\def\({\left(}\def\){\right)}

\def\sh{\mbox{\rm sh}}\def\type{\mbox{\rm type}}

\theoremstyle{plain}
\newtheorem{theorem}{Theorem}[section]
\newtheorem{lemma}[theorem]{Lemma}

\newtheorem{example}[theorem]{Example}
\newtheorem{algorithm}[theorem]{Algorithm}
\newtheorem{definition}[theorem]{Definition}

\newtheorem{corollary}[theorem]{Corollary}

\newtheorem{row bumping lemma}[theorem]{Row Bumping Lemma}

\begin{document}
\small
\normalsize
\title{Pattern Avoidance of Generalized Permutations}
\author
{
\small Zhousheng Mei, Suijie Wang\\
\small College of Mathematics and Econometrics\\
\small Hunan University, Changsha, China.\\
\small\tt zhousheng@hnu.edu.cn, wangsuijie@hnu.edu.cn
}
\date{}
\maketitle

\begin{abstract}
In this paper, we study pattern avoidances of generalized permutations and show that the number of all generalized permutations avoiding $\pi$ is independent of the choice of $\pi\in S_3$, which extends the classic results on permutations avoiding  $\pi\in S_3$. Extending both Dyck path and Riordan path, we introduce the Catalan-Riordan path which turns out to be a combinatorial interpretation of the difference array of Catalan numbers. As applications, we interpret both Motzkin and Riordan numbers in two ways, via semistandard Young tableaux of two rows and generalized permutations avoiding $\pi \in S_3$. Analogous to Lewis's method, we establish a bijection from generalized permutations to rectangular semistandard Young tableaux which will recover several known results in the literature. \vspace{2ex}
\end{abstract}
\section{Introduction}
A composition of a positive integer $N$ is a  sequence $\alpha=(\alpha_{1},\alpha_{2},\ldots,\alpha_{n})$ of nonnegative integers such that $\alpha_{1}+\cdots+\alpha_{n}=N$. Denote by $[\alpha]=\{1^{\alpha_{1}},2^{\alpha_{2}},\ldots, n^{\alpha_{n}}\}$ the multiset with multiplicity $\alpha_i$ of $i\in [n]$. Given two compositions $\alpha$ and $\beta$ of $N$, a {\em generalized permutation} of  $\alpha\to\beta$ is a one-to-one correspondence from multisets $[\alpha]$ to $[\beta]$. Denote by $S_{\alpha\beta}$ the set of all generalized permutations of  $\alpha\to\beta$. Each generalized permutation $\tau\in S_{\alpha\beta}$ can be written uniquely in two-rowed array
\[
\tau=\left(\begin{array}{cccccccccc}1&\cdots&1&2&\cdots&2&\cdots&n&\cdots&n\\ \tau_{11}&\cdots&\tau_{1\alpha_{1}}&\tau_{21}&\cdots&\tau_{2\alpha_{2}}&\cdots&\tau_{n1}&\ldots&\tau_{n
\alpha_{n}}\\
\end{array}\right)
\]
such that $\tau_{i1}\leq \tau_{i2}\leq\cdots\leq \tau_{i\alpha_{i}}$ for all $1\leq i\leq n$ and $\{\tau_{11},\tau_{12},\ldots,\tau_{n\alpha_{n}}\}=[\beta]$ as multisets.
Without causing confusions, we abbreviate $\tau=\tau_{11}\tau_{12}\ldots\tau_{n\alpha_{n}}$ and call the subsequence $\tau(i)=\tau_{i1}\tau_{i2}\ldots\tau_{i\alpha_{i}}$ the $i$-th block of $\tau$.
For convenience, write continuous repetitions of equal elements as exponential notations, e.g., the composition $(1,2,2,2,1)$ is abbreviated $(1,2^3,1)$. When $\alpha=\beta=( 1^{n})$, $S_{(1^n)(1^n)}=S_{n}$ is the set of all permutations on $[n]=\{1,2,\ldots,n\}$. As classical counting problems, pattern avoidances of permutations have been well studied, see \cite{MacMahon,Simion,West1990} etc.. In this paper, we consider the pattern avoidances of generalized permutations.
\begin{definition}
For $\tau=\tau_{11}\tau_{12}\ldots\tau_{n\alpha_{n}}\in S_{\alpha\beta}$ and $\pi\in S_m$, $\tau$ avoids the pattern $\pi$ if no index sets $1\leq i_{1}<i_{2}<\cdots<i_{m}\leq n$ and $j_{1},j_{2},\ldots ,j_{m}$ satisfy that the subsequence $\tau_{i_{1}j_{1}}\tau_{i_{2}j_{2}}\ldots\tau_{i_{m}j_{m}}$ is order isomorphic to $\pi$, i.e., $\tau_{i_{s}j_{s}}<\tau_{i_{t}j_{t}}$ iff $\pi_{s}<\pi_{t}$. Let $S_{\alpha\beta}(\pi)$ denote the set of generalized permutations in $S_{\alpha\beta}$ which avoid the pattern $\pi$.
\end{definition}

For example, for $\alpha=(1,2,2,1), \beta=(3,2,1)$, and ${\tiny\tau=\left(\begin{array}{cccccc}1&2&2&3&3&4\\ 2&1&3&1&2&1\\ \end{array}\right)\in S_{\alpha\beta}}$, we have
\[
\tau \in S_{\alpha\beta}(132)\quad \And \quad\tau\notin S_{\alpha\beta}(231),
\]
since the subsequence ${\tiny\left(\begin{array}{ccc}1&2&3 \\2&3&1\\\end{array}\right)}$ of $\tau$ is of pattern $231$. In 1973, Knuth \cite{Knuth1973} established a celebrated result on pattern avoidances of permutations.
\begin{theorem}{\em \cite{Knuth1973}} The cardinality
\[
|S_{( 1^{n})( 1^{n})}(\pi)|=C_{n}=\frac{1}{n+1}{2n\choose n}
\]
is independent of the choice of $\pi\in S_{3}$, where $C_{n}$ is the $n$-th Catalan number.
\end{theorem}
In 1985, Simion and Schmidt \cite{Simion} determined the number $|S_{(1^{n})(1^{n})}(T)|$ of permutations simultaneously avoiding any given set $T$ of patterns in $S_{3}$. In 2001, Albert, Aldred, Atkinson, Handley, and Holton \cite{Albert} calculated $|S_{( 1^{n})\beta}(T)|$ for any subset $T\subseteq S_{3}$. In 2006, using the results in \cite{Albert,Atkinson}, Savage and Wilf \cite{Savage} observed that
\begin{theorem}\label{Savage}{\em \cite{Savage}} If $\beta=(\beta_1,\ldots,\beta_m)$ is a composition of $N$, then the cardinality of $S_{( 1^{N})\beta}(\pi)$ is independent of the choice of $\pi\in S_{3}$ and the order of entries of $\beta$.
\end{theorem}
In 2007, Myers \cite{Myers} gave a bijection proof on Theorem \ref{Savage} by extending the construction of Simion and Schmidt \cite{Simion} for permutations.  Let $\lambda$ be a partition of $N$ and $\mu$ a composition of $N$. The number $K_{\lambda\mu}$ of SSYTs (semistandard Young tableaux) of shape $\lambda$ and type $\mu$ is called the Kostka number, which is independent of the order of entries of $\mu$ and Schur-concave on $\mu$, i.e., $K_{\lambda\mu}\le K_{\lambda\mu'}$ if $\mu$ dominates $\mu'$ for any two partitions $\mu$ and $\mu'$. Below we consider generalized permutations avoiding $\pi\in S_3$ and obtain a parallel result as above, whose proof is given in Section 2.
\begin{theorem}\label{main-1}
Let $\alpha=(\alpha_{1},\ldots,\alpha_{n})$ and $\beta=(\beta_{1},\ldots,\beta_{m})$ be two compositions of $N$ and denote $(\alpha,\beta)=(\alpha_{1},\ldots,\alpha_{n},\beta_{1},\ldots,\beta_{m})$. The cardinality of $S_{\alpha\beta}(\pi)$ is the Kostka number $K_{(N,N)(\alpha,\beta)}$, i.e.,
\[
|S_{\alpha\beta}(\pi)|=|S_{\beta\alpha}(\pi)|=K_{( N,N)(\alpha,\beta)}
\]
which is independent of the choice of $\pi\in S_{3}$ and the order of entries of $\alpha$ and $\beta$, and Schur-concave on both $\alpha$ and $\beta$ if $\alpha$ and $\beta$ are partitions.
\end{theorem}
It is well known that the $n$-th Catalan number $C_n$ has two classic combinatorial interpretations via SYTs (standard Young tableaux) of shape $(n^2)$ and permutations avoiding $\pi\in S_3$, i.e.,
\[
K_{( n^{2})(1^{2n})}=|S_{( 1^{n})( 1^{n})}(\pi)|=C_{n}.
\]
To the best of our knowledge, no parallel interpretations as above have been found for the Motzkin numbers  $M_n$ and Riordan numbers $R_n$.  In Section 3, we introduce the Catalan-Riordan paths which extend both Dyck and Riordan paths and turn out to be a combinatorial interpretation on the difference array of Catalan numbers.  As applications, we obtain two new combinatorial interpretations for Motzkin and Riordan numbers, for $\pi\in S_3$,
\begin{eqnarray*}
K_{( n^{2})(2^{n})}=R_{n}\quad&\And&\quad |S_{( 2^{n})( 2^{n})}(\pi)|=R_{2n},\\
K_{( (n+1)^{2})(1^2,2^{n})}=M_{n}\quad&\And&\quad |S_{( 1,2^{n})(1,2^{n})}(\pi)|=M_{2n}.
\end{eqnarray*}

In 2011, Lewis \cite{Lewis2011} established a bijection between block-ascending permutations and rectangular SYTs to enumerate two classes of block-ascending permutations via the celebrated hook length formula. In Section 4, Lewis's construction will be extended to any pair of SSYTs.
\begin{theorem}\label{Kostka theorem}
Let $\alpha=(\alpha_{1},\ldots,\alpha_{n})$ and $\beta=(\beta_{1},\ldots,\beta_{m})$ be two compositions of $N$. There is a bijection between the set $S_{\alpha\beta}$ of all generalized permutations of $\alpha\to \beta$ and the set of rectangular SSYTs of shape $(N^m)$ and type $(\alpha,N-\beta)$, i.e.,
\[
\left|S_{\alpha\beta}\right|=\sum_{\lambda\vdash N}K_{\lambda\alpha}K_{\lambda\beta}=K_{( N^{m})(\alpha,N-\beta)},
\]
where $(\alpha,N-\beta)=(\alpha_{1},\ldots,\alpha_{n},N-\beta_{1},
\ldots,N-\beta_{m})$.
\end{theorem}
In Section 5, we present some applications of above results which contain the main results of Lewis \cite{Lewis2011}, Mei and Wang \cite{Mei}, and Chen \cite{Chen}.

\section{Proof of Theorem \ref{main-1}}

\begin{lemma}\label{monotonicity-lemma}
{\rm\cite{White}} Kostka numbers $K_{\lambda\mu}$ are Schur-concave on $\mu$. Namely, if $\lambda$, $\mu$, and $\nu$ are partitions of $N$,  $\mu=(\mu_1,\mu_2,\ldots)$ and $\nu=(\nu_1,\nu_2,\ldots)$ satisfy that $\mu_{1}+\cdots+\mu_{i}\geq \nu_{1}+\cdots+\nu_{i}$ for all $i\geq 1$,  then $K_{\lambda\mu}\leq K_{\lambda\nu}$.
\end{lemma}
With the above result, Theorem \ref{main-1} is an easy consequence of three lemmas below. If $\alpha=(\alpha_{1},\alpha_{2},\ldots,\alpha_{n})$ is a composition, let $\bar{\alpha}=(\alpha_{n},\alpha_{n-1},\ldots,\alpha_{1})$ be the reverse of $\alpha$. For $\beta=(\beta_1,\beta_2,\ldots,\beta_m)$ and $\tau\in S_{\alpha\beta}$, define $\tau'\in S_{\bar{\alpha}\beta}$ and $\tau''\in S_{\alpha\bar{\beta}}$ as follows. The $i$-th block of $\tau'$ is the $(n+1-i)$-th block of $\tau$, i.e., $\tau'_{ij}=\tau_{(n+1-i)j}$ for all $1\le j\le \alpha_{n+1-i}$. The $i$-th block of $\tau''$ is defined by $\tau''_{ij}+\tau_{i(\alpha_i+1-j)}=m+1$ for all $1\le j\le \alpha_i$. \\
E.g., $\alpha=(1,1,2,1), \beta=(2,2,1)$, and $\tau=\left(\begin{array}{cccccc}1&2&3&3&4\\ 2&3&1&2&1\\ \end{array}\right)\in S_{\alpha\beta}$, then $\bar{\alpha}=(1,2,1,1), \bar{\beta}=(1,2,2)$,
\[
\tau'=\left(\begin{array}{cccccc}1&2&2&3&4\\ 1&1&2&3&2\\ \end{array}\right)\in S_{\bar{\alpha}\beta} \quad\And\quad\tau''=\left(\begin{array}{cccccc}1&2&3&3&4\\ 2&1&2&3&3\\ \end{array}\right)\in S_{\alpha\bar{\beta}}.
\]
Below is a quick fact.
\begin{lemma}
\begin{eqnarray*}
|S_{\alpha\beta}(123)|=|S_{\bar{\alpha}\beta}(321)|,  \quad\quad\quad\quad\quad~~\\
|S_{\alpha\beta}(213)|=|S_{\bar{\alpha}\beta}(312)|=
|S_{\bar{\alpha}\bar{\beta}}(132)|=|S_{\alpha\bar{\beta}}(231)|.
\end{eqnarray*}
\end{lemma}

Given two compositions $\alpha=(\alpha_1,\ldots,\alpha_n)$ and $\beta=(\beta_1,\ldots,\beta_m)$ of $N$. Let $P$ and $Q$ be SSYTs of the same shape $\lambda=(\lambda_1,\lambda_2)$ (at most two rows) and of types $\alpha$ and $\beta$ respectively. Denote by $P\boxplus Q$ a rectangular SSYT of shape $(N,N)$ and type $(\alpha,\bar{\beta})$ whose $(i,j)$-entry is
\[
P\boxplus Q(i,j)=\left\{
  \begin{array}{ll}
   P(i,j) & \If j\le \lambda_i;\\
    n+m+1-Q(3-i,N+1-j) &\;{\rm else.}
  \end{array}
\right.
\]
Namely, $P\boxplus Q$ is obtained by rotating $Q$ by $180^{\circ}$, replacing each entry $i$ of $Q$ with $n+m+1-i$, and jointing the resulting diagram with $P$. E.g., If $P={\tiny\young(113,24)}$ and $Q={\tiny\young(134,25)}$, then $P\boxplus Q={\tiny \young(11358,24679)}$. \\

For any $\tau\in S_{\alpha\beta}(321)$, if $\tau\xrightarrow{RSK}(P,Q)$, then $P$ and $Q$ have at most two rows. It is clear that the composition of $\boxplus$ and RSK correspondence
\[
\tau\,\xrightarrow{RSK}\,(P,Q)\,\longrightarrow\, P\boxplus Q
\]
gives a bijection from $S_{\alpha\beta}(321)$ to the set of all rectangular SSYTs of shape $(N, N)$ and type $(\beta,\bar{\alpha})$. Thus, $|S_{\alpha\beta}(321)|=K_{(N,N)(\beta,\bar{\alpha})}=K_{(N,N)(\alpha,\beta)}$.

\begin{lemma}\label{321-lemma}
For two compositions $\alpha$ and $\beta$ of $N$, the cardinality
\[
|S_{\alpha\beta}(321)|=|S_{\beta\alpha}(321)|=K_{(N,N)(\alpha,\beta)}
\]
is independent of the order of the entries of $\alpha$ and $\beta$, and Schur-concave on $\alpha$ and $\beta$ respectively if they are partitions.
\end{lemma}
\begin{lemma}\label{key-lemma}
For any two compositions $\alpha$ and $\beta$ of $N$, we have
\[
|S_{\alpha\beta}(213)|=|S_{\alpha\beta}(123)|.
\]
\end{lemma}
\begin{proof}Suppose $\alpha=(\alpha_1,\alpha_2,\ldots,\alpha_n)$ and $\beta=(\beta_1,\beta_2,\ldots,\beta_m)$ are compositions of $N$. The result is trivial for $n=1$. For $n\ge 2$,  we will show that $|S_{\alpha\beta}(213)|$ and $|S_{\alpha\beta}(123)|$ have the same recursive formula
\begin{eqnarray}\label{213-equation}
|S_{\alpha\beta}(\pi)|=\sum_{k=1}^{n}
\left(|S_{\alpha'(k)\beta'}(\pi)|-|S_{\alpha''(k)\beta''}(\pi)|\right) \quad\quad\text{for}~\pi=123~\text{or}~213,
\end{eqnarray}
where $|S_{\alpha\beta}(\pi)|=0$ if some entry of $(\alpha,\beta)$ is negative, $\alpha_{0}=0$, and for $1\le k\le n$,
\begin{eqnarray*}
\alpha'(k)&=&(\alpha_0+\alpha_{1}+\cdots+\alpha_{k-1},\alpha_{k}-1,\alpha_{k+1},\ldots,\alpha_{n}),\\
\alpha''(k)&=&(\alpha_0+\alpha_{1}+\cdots+\alpha_{k-1}-1,\alpha_{k}-1,\alpha_{k+1},\ldots,\alpha_{n}),\\
\beta'&=&(\beta_{1},\ldots,\beta_{m-1},\beta_{m}-1),\\ \beta''&=&(\beta_{1},\ldots,\beta_{m-1},\beta_{m}-2).
\end{eqnarray*}
Given $\tau=\tau_{11}\tau_{12}\ldots\tau_{1\alpha_{1}} \ldots \tau_{n1}\ldots\tau_{n(\alpha_{n}-1)}\tau_{n\alpha_{n}}\in S_{\alpha\beta}(123)$, let $k\in[n]$ be the minimum index such that $\tau_{k\alpha_{k}}=m$. If $k=1$, let $\tau'$ be the generalized permutation obtained from $\tau$ by removing $\tau_{1\alpha_{1}}$. Then $\tau'\in S_{\alpha'(1)\beta'}$ and avoids the pattern $123$. It gives a bijection between $\{\tau\in S_{\alpha\beta}(123)\mid \tau_{1\alpha_1}=m\}$ and $S_{\alpha'(1)\beta'}(123)$. Note $|S_{\alpha''(1)\beta''}(123)|=0$ by assumptions. Thus
\[
\#\{\tau\in S_{\alpha\beta}(123)\mid \tau_{1\alpha_1}=m\}=|S_{\alpha'(1)\beta'}(123)|-|S_{\alpha''(1)\beta''}(123)|.
\]
If $k\geq 2$, then $\tau_{i\alpha_{i}}<\tau_{k\alpha_{k}}=m$ for all $1\leq i<k$.  Since $\tau$ avoids the pattern $123$, the subsequence $\tau_{11}\tau_{12}\ldots\tau_{1\alpha_{1}}\ldots\tau_{(k-1)1}\ldots\tau_{(k-1)\alpha_{k-1}}$ of the first $k-1$ blocks avoids the pattern $12$, which implies
 \[
\tau_{(k-1)1}\leq\cdots\leq\tau_{(k-1)\alpha_{k-1}}\leq\cdots\leq \tau_{21}\leq\cdots\leq \tau_{2\alpha_{2}}\leq \tau_{11}\leq \cdots\leq \tau_{1\alpha_{1}}<m=\tau_{k\alpha_k}.
\]
Let $\gamma=(\gamma_1,\ldots,\gamma_{n-k+2})=(\alpha_1+\cdots+\alpha_{k-1},\alpha_k,\ldots,\alpha_n)$. The sequence
\[
\tilde{\tau}=\tau_{(k-1)1}\ldots\tau_{(k-1)\alpha_{k-1}}\ldots \tau_{21}\ldots \tau_{2\alpha_{2}} \tau_{11} \ldots \tau_{1\alpha_{1}}\tau_{k1}\ldots\tau_{k\alpha_k}\ldots\tau_{n1}\ldots\tau_{n\alpha_n}
\]
becomes a generalized permutation in
\[
X=\{\sigma\in S_{\gamma\beta}(123)\mid \sigma_{1\gamma_1}<\sigma_{2\gamma_2}=m\}.
\]
Indeed, the map $\tau\mapsto \tilde{\tau}$ defines a bijection in this situation, i.e., for $k\ge 2$
\[
|X|=\#\{\tau\in S_{\alpha\beta}(123)\mid \tau_{i\alpha_{i}}<\tau_{k\alpha_{k}}=m {\rm ~for~ } 1\leq i<k\}.
\]
Let
\begin{eqnarray*}
X'=\{\sigma\in S_{\gamma\beta}(123)\mid \sigma_{2\gamma_2}=m\}\quad \And \quad
X''=\{\sigma\in S_{\gamma\beta}(123)\mid \sigma_{1\gamma_1}=\sigma_{2\gamma_2}=m\}.
\end{eqnarray*}
Obviously, $|X|=|X'|-|X''|$. Similar as before, we can construct two bijections, between $X'$ and $S_{\gamma'(2)\beta'}(123)$ by removing $\sigma_{2\gamma_2}$ from $\sigma\in X'$, and between $X''$ and $S_{\gamma''(2)\beta''}(123)$ by removing $\sigma_{1\gamma_1},\sigma_{2\gamma_2}$ from $\sigma\in X''$. It follows that in this case,
\[
|X|=|S_{\gamma'(2)\beta'}(123)|-|S_{\gamma''(2)\beta''}(123)|.
\]
Notice
\[
S_{\gamma'(2)\beta'}(123)=S_{\alpha'(k)\beta'}(123)\quad \And \quad S_{\gamma''(2)\beta''}(123)=S_{\alpha''(k)\beta''}(123).
\]
Hence
\[
\#\{\tau\in S_{\alpha\beta}(123)\mid \tau_{i\alpha_{i}}<\tau_{k\alpha_{k}}=m {\rm~for~ } 1\leq i<k\}=|S_{\alpha'(k)\beta'}(123)|-|S_{\alpha''(k)\beta''}(123)|.
\]
So far, we have proved that $S_{\alpha\beta}(123)$ satisfies the recursive formula (\ref{213-equation}). For $S_{\alpha\beta}(213)$, the proof is similar. For $\tau\in S_{\alpha\beta}(213)$, let $k\in [n]$ be the minimum index such that $\tau_{k\alpha_{k}}=m$. If $k=1$, the proof is the same as before. If $k\geq 2$, then $\tau_{i\alpha_{i}}<\tau_{k\alpha_{k}}=m$ for all $1\leq i<k$.  Since $\tau$ avoids the pattern $213$, the subsequence $\tau_{11}\tau_{12}\ldots\tau_{(k-1)1}\ldots\tau_{(k-1)\alpha_{k-1}}$ avoids the pattern $21$, which implies
\[
\tau_{11}\le\cdots\le \tau_{1\alpha_1}\le\tau_{21}\le \cdots\le \tau_{(k-1)1}\le \cdots\le \tau_{(k-1)\alpha_{k-1}}<m=\tau_{k\alpha_k}.
\]
Then $\tau$ can be regarded as a generalized permutation in $X$. Applying the same arguments as before, we can obtain that $S_{\alpha\beta}(213)$ also has the recursive formula (\ref{213-equation}). The proof completes by induction on $N$.
\end{proof}

\section{Catalan-Riordan Paths}
A \emph{Dyck path} of length $2n$ is a lattice path in the $xy$-plane from $(0,0)$ to $(2n,0)$ with the step set $\{(1,1),(1,-1)\}$ and never going below the $x$-axis. The number of Dyck paths of length $2n$ is the $n$-th Catalan number $C_n$. Among those combinatorial interpretations of Catalan numbers, two  classic interpretations are via SYTs and pattern avoidance of permutations, i.e.,
\begin{equation}\label{Catalanformula}
K_{(n^2)(1^n)}=S_{(1^n)(1^n)}(\pi)=C_n=\frac{1}{n+1}{2n\choose n}.
\end{equation}
A \emph{Motzkin path} of length $n$ is a lattice path in the $xy$-plane from $(0,0)$ to $(n,0)$ with the step set $\{(1,1),(1,0),(1,-1)\}$ that never goes below the $x$-axis. A \emph{Riordan path} of length $n$ is a Motzkin path from $(0,0)$ to $(n,0)$ that has no step $(1,0)$ on the $x$-axis. The number of Motzkin and Riordan paths of length $n$ is the $n$-th Motzkin number $M_n$ and Riordan number $R_n$ respectively. There are some interpretations on Motzkin and Riordan numbers as SYTs and pattern avoidances of permutations. For Motzkin numbers, Zabrocki interpreted $M_n$ as the number of SYTs of height $\le 3$, Baril \cite{Baril2011} obtained that $M_n$ is the number of $(n+1)$- length permutations avoiding the pattern $132$ and the dotted pattern $23\dot{1}$, Burstein and Pantone \cite{Burstein2014} showed that $M_n$ is the number of $n$-length involutions avoiding patterns $4231$ and $5276143$. For Riordan numbers, Callan obtained that $R_n$ is the number of $321$-avoiding permutations on $[n]$ in which each left-to-right maximum is a descent,
Chen-Deng-Yang \cite{Chen YC} proved that $R_n$ is  the number of derangements on $[n]$ that avoid both $321$ and $3\bar{1}42$, Regev \cite{Regev2010} found that $R_{n}$ is the number of SYTs of shape $( k,k,1^{n-2k})$ for all $k\geq 0$ which gave an expression of $R_n$ through the hook length formula. See OEIS \cite[A001006, A005043]{Sloane} for more investigations and combinatorial interpretations of Motzkin and Riordan numbers.

In this section, we will introduce the Catalan-Riordan path, which extends the concepts of Dyck and Riordan paths. Parallel to the interpretations (\ref{Catalanformula}) on Catalan numbers , we obtain two interpretations of Motzkin and Riordan numbers via SSYTs and pattern avoidances respectively, which are new in our knowledge. In 1999, Bernhart \cite{Bernhart} studied the difference array of Catalan numbers and gave the following formulae on Catalan, Riordan, and Motzkin numbers.
\begin{theorem}\label{Bernhart}
{\rm\cite{Bernhart}}
If $C_{n}$, $M_n$, and $R_{n}$ are the $n$-th Catalan, Motzkin, and Riordan numbers respectively, then
\[
C_{n}=\sum_{i=0}^{n}{n\choose i}R_{i},\quad R_{n}=\sum_{i=0}^{n}(-1)^{n-i}{n\choose i}C_{i}, \quad M_n=R_n+R_{n+1}.
\]
\end{theorem}

\begin{definition}[Catalan-Riordan Path]\label{generalized Riordan}
For $0\le k\le n$, let $\mathcal{CR}(n,k)$ be the set of lattice paths of length $n+k$ in the $xy$-plane from $(0,0)$ to $(n+k,0)$ such that
\begin{enumerate}[{\rm (1)}]
\item the first $2k$ steps have the step set  $\{(1,1),(1,-1)\}$;
\item the last $n-k$ steps have the step set $\{(1,2),(1,0),(1,-2)\}$, and no step $(1,0)$ on  the $x$-axis;
\item never go below the $x$-axis.
\end{enumerate}
Members of $\mathcal{CR}(n,k)$  are called Catalan-Riordan paths of size $(n,k)$, and the cardinality of $\mathcal{CR}(n,k)$ is called the $(n,k)$-th Catalan-Riordan number, denoted $CR(n,k)$.
\end{definition}
Note that $\mathcal{CR}(n,n)$ is the set of Dyck paths from $(0,0)$ to $(2n,0)$. When $k=0$, replacing all steps $(1,2)$ and $(1,-2)$ with steps $(1,1)$ and $(1,-1)$ respectively, $\mathcal{CR}(n,0)$ becomes the set of Riordan paths from $(0,0)$ to $(n,0)$. Namely,
\[
CR(n,0)=R_n\quad\And\quad CR(n,n)=C_n.
\]
\begin{theorem}\label{main-2}
There is a bijection between $\mathcal{CR}(n,k)$ and the set of SSYTs of shape $( n,n)$ and type $( 1^{2k},2^{n-k})$, i.e.,
\[
CR(n,k)=K_{(n,n)(1^{2k},2^{n-k})}.
\]
\end{theorem}
\begin{proof}
Let $T=(t_{ij})$ be a SSYT of shape $( n^{2})$ and type $( 1^{2k},2^{n-k})$. For $1\le s\le n+k$, suppose the number of $s$ in the first row of $T$ is $y_s$ more than the second row, i.e.,
\[
y_s=\#\{j\in [n]\mid t_{1j}=s\}-\#\{j\in [n]\mid t_{2j}=s\}.
\]
Obviously, $y_s=1,-1$ if $1\le s\le 2k$ and $y_s=2,0,-2$ otherwise. We construct a lattice path $L$ starting at $(0,0)$ such that the $s$-th step of $L$ is $(1,y_s)$ for all $1\le s\le n+k$. We claim that this construction defines a map from the set of SSYTs of shape $( n,n)$ and type $( 1^{2k},2^{n-k})$ to $\mathcal{CR}(n,k)$. Note that all entries $t_{ij}$ with $t_{ij}\le s$ appearing in $T$ form a SSYT whose second row is no longer than the first row, which implies $y_1+\cdots+y_s\ge 0$, i.e., the path $L$ never goes below the $x$-axis. To prove the path $L$ has no step $(1,0)$ on the $x$-axis, suppose $L$ goes to the $x$-axis at step $s\ge2k$, i.e., $y_1+\cdots+y_s=0$. From the definition of $y_s$, the entries $t_{ij}$ with $t_{ij}\le s$ appearing in $T$ form a SSYT whose two rows have the same length. It forces that both $s+1$ must lie in the first row of $T$, i.e., $y_{s+1}=2$. Since two rows of $T$ have the same length, it follows by similar arguments that the path $L$ ends at $(n+k,0)$. One can easily obtain the inverse map by inverting the above constructions and prove that it is well-defined.
\end{proof}
\begin{corollary}[Recursive Formula]\label{Generalized-R-pro}For $1\le k\le n$,
 \begin{eqnarray*}
 CR(n,k)=CR(n,k-1)+CR(n-1,k-1).
\end{eqnarray*}
In particular, $CR(n+1,1)=R_{n+1}+R_{n}=M_{n}$.
\end{corollary}
\begin{proof}
From Theorem \ref{main-2},  $CR(n,k)=K_{(n,n)(1^{2k},2^{n-k})}$. Let
\begin{eqnarray*}
\mathcal{T}_{n,k}&=&\{T\mid T~\text{is~a~SSYT~of~shape}~(n,n)~\text{and~type}~(1^{2k},2^{n-k})\}\\
&=&\{T=(t_{ij})\in \mathcal{T}_{n,k}\mid t_{12}=2\}\sqcup\{T=(t_{ij})\in \mathcal{T}_{n,k}\mid t_{21}=2\}.
\end{eqnarray*}
For $T=(t_{ij})\in \mathcal{T}_{n,k}$ with $t_{12}=2$, replacing $t_{12}=1$, one obtains a SSYT of shape $(n,n)$ and type $(2,0,1^{2(k-1)},2^{n-k})$. Thus
\[
\#\{T=(t_{ij})\in \mathcal{T}_{n,k}\mid t_{12}=2\}=K_{(n,n)(1^{2(k-1)},2^{n-(k-1)})}.
\]
For $T=(t_{ij})\in \mathcal{T}_{n,k}$ with $t_{21}=2$, removing $t_{11}$ and $t_{21}$ from $T$, one obtains a SSYT of shape $(n-1,n-1)$ and type $(1^{2(k-1)},2^{n-k})$. Thus
\[
\#\{T=(t_{ij})\in \mathcal{T}_{n,k}\mid t_{21}=2\}=K_{(n-1,n-1)(1^{2(k-1)},2^{n-k})}.
\]
 By Theorem \ref{main-2}, the result holds.
\end{proof}
Indeed, similar as above arguments, if $\alpha=(\alpha_1,\ldots,\alpha_m)$ is a composition of $2n$, we can obtain a recursive formula for $K_{(n^2)\alpha}$,
\[
K_{(n^2)\alpha}=\sum_{i=0}^{\min\{\alpha_1,\alpha_2\}}K_{(n^2)
(\alpha_1+\alpha_2-2i,\alpha_3,\ldots,\alpha_m)}.
\]
Let $D(f_n)=f_n-f_{n-1}$ be the difference of the sequence $f_n$, and $D^k(f_n)=D(D^{k-1}(f_n))$ the $k$-th difference of $f_n$ for $k\ge 2$. It is easily seen that
\[
CR(n,k)=D^{n-k}(C_n),
\]
which gives a combinatorial interpretation on the difference array of Catalan numbers, see OEIS \cite[A059346]{Sloane}.
\begin{corollary}\label{CR formula}If $0\le k\le n$, then
 \[
 CR(n,k)=\sum_{i=0}^{k}{k\choose i}R_{n-i}=\sum_{i=0}^{n-k}(-1)^{i}{n-k\choose i}C_{n-i}.
 \]
 In particular, $C_{n}=\sum_{i=0}^{n}{n\choose i}R_{i}$ and $R_n=\sum_{i=0}^{n}(-1)^{i}{n\choose i}C_{n-i}$.
\end{corollary}
\begin{proof}
From Corollary \ref{Generalized-R-pro}, we have for any $1\le l\le k\le n$,
\[
CR(n,k)=\sum_{i=0}^{l}{l\choose i}CR(n-i,k-l).
\]
Recall that $CR(n,0)=R_{n}$, taking $l=k$ in above equation, we have $CR(n,k)=\sum_{i=0}^{k}{k\choose i}R_{n-i}$. Since $CR(n,n)=C_{n}$, we have
\[
C_{n}=\sum_{i=0}^{l}{l\choose i}CR(n-i,n-l)\quad \text{for~all}~n\ge 1.
\]
Via the inclusion-exclusion principle, we have
\[
CR(n,n-l)=\sum_{i=0}^{l}(-1)^{i}{l\choose i}C_{n-i}.
\]
\end{proof}
As an easy consequence of Theorem \ref{main-1}, Theorem \ref{main-2}, and Corollary \ref{CR formula}, below we obtain three combinatorial interpretations on the difference array of Catalan numbers, via lattice paths, SSYTs, and pattern avoidances. Note $CR(n+1,1)=M_{n}$ and $CR(n,0)=R_n$. Analogous to Catalan numbers (\ref{Catalanformula}), we present two new combinatorial interpretations for Morzkin and Riordan numbers.
\begin{theorem}
Let $\alpha=(1^{m_1},2^{n_1})$ and $\beta=(1^{m_2},2^{n_2})$ be two compositions of $n$, and assume $m_1+m_2=2k$.
We have for any $\pi\in S_3$,
\[
D^{n-k}(C_n)=CR(n,k)=K_{(n,n)(1^{2k},2^{n-k})}=|S_{\alpha\beta}(\pi)|.
\]
In particular,
\begin{eqnarray}
M_n=K_{((n+1)^{2})(1^2,2^n)}\quad&\And& \quad R_n=K_{( n^{2})(2^n)},\nonumber\\
M_{2n}=|S_{(1,2^n)(1,2^n)}(\pi)|\quad&\And&\quad R_{2n}=|S_{(2^n)(2^n)}(\pi)|.\label{bijections}
\end{eqnarray}
\end{theorem}
In \cite[Section 1.5]{Stanley-Vol1}, some classic geometric constructions on permutation matrices  can establish a direct bijection between $S_n(\pi)$ and Dyck paths of length $2n$ for each $\pi\in S_3$. It would be really interesting if we can find analogous constructions on integral matrices to obtain geometric bijections from $S_{(1,2^n)(1,2^n)}(\pi)$ to Motzkin paths of length $2n$, and from $S_{(2^n)(2^n)}(\pi)$ to Riordan paths of length $2n$.

\section{Extension of Lewis's Construction}
Let  $\alpha=(\alpha_{1},\ldots,\alpha_{n})$ and $\beta=(\beta_{1},\ldots,\beta_{m})$ be two compositions of $N$. In the rest of this paper, $m$ and $n$ can be any integer $m'$ and $n'$ with $m'\ge m$ and $n'\ge n$ respectively only if we take $\alpha_i=\beta_j=0$ for $n<i\le n'$ and $m<j\le m'$.  For any $d\in\Bbb{N}$, denote
\begin{eqnarray*}
\mathfrak{S}_{d}(\alpha,\beta)&=&\{(P,Q)\mid \sh(P)=\sh(Q)\subseteq (d^m), \type(P)=\alpha, \type(Q)=\beta \},\\
\mathfrak{T}_{d}(\alpha,\beta)&=&\{T\mid \sh(T)=(d^{m}), \type(T)=(\alpha,d-\bar{\beta})\},
\end{eqnarray*}
where $\sh(P)$ and $\type(P)$ are the shape and type of $P$, $d-\bar{\beta}=(d-\beta_m,\dots, d-\beta_1)$. Define a map
\[
\theta:\; \mathfrak{S}_{d}(\alpha,\beta)\to\mathfrak{T}_{d}(\alpha,\beta)
\]
such that if $P=(P_{ij})$ and $Q=(Q_{ij})$ are two SSYTs of the same shape $\lambda=(\lambda_1,\lambda_2,\ldots)$ with $\lambda_1\le d$ and of types $\alpha$ and $\beta$ respectively, then $\theta(P,Q)=T=(T_{ij})$ is defined as follows, for $1\le k\le m$,
\begin{eqnarray*}\label{theta_d}
T_{ij}=\left\{
  \begin{array}{ll}
   P_{ij}& \quad \If 1\le i\le l(\lambda), \; 1\le j\le \lambda_i; \\
   n+k&\quad\If k\le i\le m, \; \lambda_{i-k+1}(m-k+1)<j\le \lambda_{i-k}(m-k),
  \end{array}
\right.
\end{eqnarray*}
where $\lambda_0(k)=d$ and $\lambda_i(k)$ is the maximal $j$ such that $Q_{ij}\le k$, i.e.,
\begin{equation*}\label{lambda_i(k)}
\lambda_i(k)=\#\{j\mid Q_{ij}\le k \}.
\end{equation*}
Indeed, from the proof of Theorem \ref{bijection on SSYT}, we will see that the map $\theta$ can be realized by the following algorithm which extends those constructions by Lewis \cite{Lewis2011}, Mei and Wang \cite{Mei}.
\begin{algorithm}\label{algorithm}
We start with a Young diagram of shape $(d^{m})$ whose upper-left corner of shape $\lambda$ is filled with $P$. Next we fill the remaining empty boxes of shape $(d^{m})/\lambda$ by $m$ steps. At the first step, we fill the number $n+m$ into each box of the rightmost $d-\beta_1$ boxes of the bottom row. Generally at the $i$-th step for $2\le i\le m$, for those remaining empty boxes of the diagram $(d^{m})$ after step $i-1$, the bottom one of those $d-\beta_i$ columns where $Q$ contains no $i$ is filled by the entry $m+n+1-i$. After the $m$-th step, we obtain $\theta(P,Q)$.
\end{algorithm}
\begin{example}
Let $P$ and $Q$ be two SSYTs  with  $\sh(P)=\sh(Q)=(4,3,2)$, $\type(P)=(2,2,3,2)$, and  $\type(Q)=(3,2,1,1,2)$ as follows,
\[
\begin{array}{c}P=\\{}\\{}\\{}\\{}\\\end{array}
\young(1123,234,34)\hspace{2mm}
\begin{array}{c}and\\{}\\{}\\{}\\{}\\\end{array}
\begin{array}{c}Q=\\{}\\{}\\{}\\{}\\\end{array}
\young(1115,224,35).\vspace{-10mm}
\]
Taking $d=5$, then we fill the Young diagram of shape $(5^{5})$ step by step as follows,
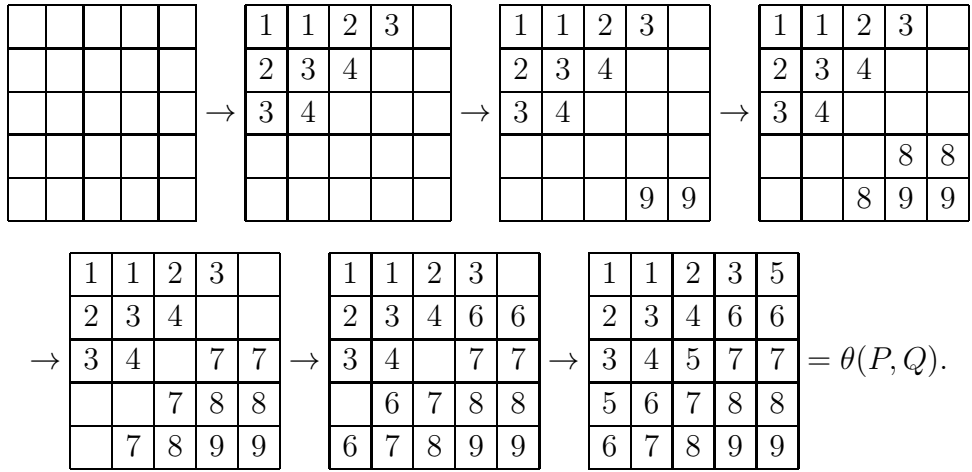
\[
\begin{array}{|c|c|c|c|c|c|c|}\hline ~&~ &~ &~&~ \\\hline ~& ~&~ &~ &~ \\\hline ~& ~&~ &~ &~\\\hline ~& ~& ~&~&~ \\\hline ~&~ &~ &~ &~\\\hline\end{array}
\rightarrow
\begin{array}{|c|c|c|c|c|c|c|}\hline 1&1&2&3&~ \\\hline 2& 3&4 &~& \\\hline 3&4 & & & \\\hline & & ~&~& \\\hline & & &&\\\hline\end{array}
\rightarrow
\begin{array}{|c|c|c|c|c|c|c|}\hline 1&1&2&3& \\\hline 2& 3& 4& &\\\hline 3&4 & && \\\hline & & && \\\hline & & &9&9\\\hline\end{array}
\rightarrow
\begin{array}{|c|c|c|c|c|c|c|}\hline 1& 1 &2 &3 &\\\hline 2& 3& 4&& \\\hline 3& 4& &&\\\hline & & &8&8 \\\hline & & 8&9&9\\\hline\end{array}
\]
\[
\rightarrow
\begin{array}{|c|c|c|c|c|c|c|}\hline 1 & 1 &2 &3 &\\\hline 2 & 3& 4&&\\\hline3 &4 & &7&7\\\hline & & 7&8&8 \\\hline & 7& 8&9&9\\\hline\end{array}
\rightarrow
\begin{array}{|c|c|c|c|c|c|c|}\hline 1 & 1 & 2& 3&\\\hline 2 &3 &4 &6&6 \\\hline 3 & 4&  & 7&7\\\hline & 6&7 &8&8\\\hline 6& 7&8 &9&9\\\hline\end{array}
\rightarrow
\begin{array}{|c|c|c|c|c|c|c|}\hline 1 & 1 & 2& 3&5\\\hline 2 &3 &4 &6 &6 \\\hline 3 & 4& 5 & 7&7\\\hline 5& 6&7 &8&8\\\hline 6& 7&8 &9&9\\\hline\end{array}=\theta(P,Q).
\]
On the other hand, by the definition of $\lambda_i(k)$, we have
\begin{eqnarray*}
&&\lambda_1(1)=\lambda_1(2)=\lambda_1(3)=\lambda_1(4)=3,\lambda_1(5)=4,\\
&&\lambda_2(1)=0, \lambda_2(2)=\lambda_2(3)=2,\lambda_2(4)=\lambda_2(5)=3,\\
&&\lambda_3(1)=\lambda_3(2)=0,\lambda_3(3)=\lambda_3(4)=1,\lambda_3(5)=2,\\
&&\lambda_0(k)=5\quad \And \quad\lambda_4(k)=\lambda_5(k)=0.
\end{eqnarray*}
By the definition of $\theta$, we have
$T_{ij}=4+k$ iff  $\lambda_{i-k+1}(5-k+1)<j\le \lambda_{i-k}(5-k).$
So
\begin{eqnarray*}
T_{ij}=5 \quad{\rm iff}\quad \lambda_{i}(5)<j\le \lambda_{i-1}(4)\quad&\Longrightarrow&\quad T_{15}=T_{33}=T_{41}=5,\\
T_{ij}=6 \quad{\rm iff}\quad \lambda_{i-1}(4)<j\le \lambda_{i-2}(3)\quad&\Longrightarrow&\quad T_{51}=T_{42}=T_{24}=T_{25}=6,\\
T_{ij}=7 \quad{\rm iff}\quad \lambda_{i-2}(3)<j\le \lambda_{i-3}(2)\quad&\Longrightarrow&\quad T_{52}=T_{43}=T_{34}=T_{35}=7,\\
T_{ij}=8 \quad{\rm iff}\quad \lambda_{i-3}(2)<j\le \lambda_{i-4}(1)\quad&\Longrightarrow&\quad T_{53}=T_{44}=T_{45}=8,\\
T_{ij}=9 \quad{\rm iff}\quad \lambda_{i-4}(1)<j\le \lambda_{i-5}(0)\quad&\Longrightarrow&\quad T_{54}=T_{55}=9.
\end{eqnarray*}
\end{example}

\begin{theorem}\label{bijection on SSYT}
The map $\theta$ is a bijection.
\end{theorem}
\begin{proof}First we prove that $\theta$ is well-defined. Suppose $\theta(P,Q)=T=(T_{ij})$ for $P, Q\in \mathfrak{S}_{d}(\alpha,\beta)$. We need to show that $T$ is a SSYT of shape $(d^m)$ and type $(\alpha,d-\bar{\beta})$.
By the definition of $\lambda_i(k)$, we have $Q_{(i+1)(\lambda_{i}(k)+1)}> Q_{i(\lambda_i(k)+1)}\ge k+1\ge Q_{(i+1)\lambda_{i+1}(k+1)}$
which implies $\lambda_{i}(k)+1> \lambda_{i+1}(k+1)$, i.e., \[\lambda_{i}(k)\ge \lambda_{i+1}(k+1).\] Given $k\le i\le m$, since
\begin{eqnarray}
\{j\mid T_{ij}=n+k\}&=&\{j\mid \lambda_{i-k+1}(m-k+1)<j\le \lambda_{i-k}(m-k)\},\nonumber\\
\{j\mid T_{ij}=n+k+1\}&=&\{j\mid \lambda_{i-k}(m-k)<j\le \lambda_{i-k-1}(m-k-1)\},\label{T_ij}\\
......\quad\quad\quad&&\nonumber
\end{eqnarray}
we have $\{j\mid T_{ij}\ge n+k\}=\{j\mid j> \lambda_{i-k+1}(m-k+1)\}$. Thus for $1\le i\le m-k+1$,
\[
\{j\mid T_{(i+k-1)j}\ge n+k\}=\{j\mid \lambda_{i}(m-k+1)<j\le d\}.
\]
Taking $k=1$, we obtain
\[
\{(i,j)\mid T_{ij}\ge n+1\}=\{(i,j)\mid \lambda_{i}(m)=\lambda_i<j\le d\},
\]
which implies that the shape of $T$ is $(d^{m})$. To prove that the type of $T$ is $(\alpha,d-\bar{\beta})$, since $\lambda_i(k)$ is the maximal $j$ such that $Q_{ij}\le k$, it follows that $Q_{ij}=k$ iff $\lambda_{i}(k-1)<j\le \lambda_i(k)$. Thus for $1\le i\le m-k+1$,
\[
\{j\mid Q_{ij}=m-k+1\}=\{j\mid\lambda_{i}(m-k)< j\le \lambda_{i}(m-k+1)\}.
\]
By (\ref{T_ij}),
\[
\{j\mid T_{(i+k-1)j}=n+k\}=\{j\mid \lambda_{i}(m-k+1)<j\le \lambda_{i-1}(m-k)\}.
\]
Then for $1\le i\le m-k+1$, we have
\[
\#\{j\mid Q_{ij}=m-k+1\}+\#\{j\mid T_{(i+k-1)j}=n+k\}=\lambda_{i-1}(m-k)-\lambda_i(m-k).
\]
Note that $\lambda_i(k)=0$ for $i>k$. Then $\sum_{i=1}^m(\lambda_{i-1}(m-k)-\lambda_i(m-k))=d$. For $i>m-k+1$, there is no $j\in [d]$ such that $Q_{ij}=m-k+1$ and $T_{(i+k-1)j}=n+k$.
If each pair $(i,j)$ with $T_{(i+k-1)j}=n+k$ for $1\le i\le m-k+1$ is substituted by the pair $(i,j)$ with $T_{ij}=n+k$ for $k\le i\le m$, then
\[
\#\{(i,j)\mid T_{ij}=n+k\}+\#\{(i,j)\mid Q_{ij}=m-k+1\}=d.
\]
So $T$ is of shape $(\alpha,d-\bar{\beta})$. For any $k\le i\le m$, it follows by (\ref{T_ij}) that in each row of $T$, the boxes with entry $n+k+1$ are strictly right of the boxes with entry $n+k$, i.e., $T$ is weakly increasing across each row. It remains to show that $T$ is strictly increasing down each column. Suppose there exist $i$ and $j$ such that $T_{ij}=n+k$ and $T_{(i+1)j}=n+k'$ with $k>k'$. Since $\lambda_{i}(k)\ge \lambda_{i+1}(k+1)$, we have
\[
\lambda_{i-k'+1}(m-k')\le\lambda_{i-k+1}(m-k)\le \lambda_{i-k+1}(m-k+1).
\]
Via the definition of $T_{ij}$,
\[
j\in \{j\mid\lambda_{i-k+1}(m-k+1)<j\le \lambda_{i-k}(m-k)\}\cap\{j\mid\lambda_{i-k'+2}(m-k'+1)<j\le \lambda_{i-k'+1}(m-k')\},
\]
which is a contradiction. Thus $T\in \mathfrak{T}_{d}(\alpha,\beta)$ and the map $\theta$ is well-defined.

To prove $\theta$ is a bijection, it is enough to show that for any $T\in \mathfrak{T}_{d}(\alpha,\beta)$, there is a unique pair $(P,Q)\in \mathfrak{S}_{d}(\alpha,\beta)$ such that $\theta(P,Q)=T$. Given $T\in \mathfrak{T}_{d}(\alpha,\beta)$, define $P=(P_{ij})$ as
\[
P_{ij}=T_{ij} \quad\If 1\le T_{ij}\le n.
\]
It is obvious that $P$ is uniquely determined by $T$. Since $T$ is of shape $(d^m)$ and type $(\alpha,d-\bar{\beta})$, we have  $\sh(P)\subseteq (d^m)$ and $\type(P)=\alpha$. Denote by $\lambda=(\lambda_1,\lambda_2,\ldots)$ the shape of $P$ and assume $\lambda_i=0$ if $i> l(\lambda)$.
For $i\in [m]$ and $s\in [m+n]$, let
\[
\mu_{i}(s)=\#\{j\mid T_{ij}\le s\}.
\]
If $i\in [m]$, $\mu_{i}(n)=\lambda_{i}$ and $\mu_i(n+k)=d$ for $i\le k$. Let $\mu_{i}(s)=0$ if $i>m$.
Now we construct $Q=(Q_{ij})$ as follows, for $k\in [m]$,
\[
Q_{ij}=m-k+1 \quad\If 1\le i\le m-k+1,\;\mu_{k+i}(n+k)<j\le\mu_{k+i-1}(n+k-1).
\]
To prove that $Q$ is a SSYT of shape $\lambda$ and type $\beta$, the arguments are similar as before. Since $T_{(i+1)(\mu_{i}(s)+1)}>T_{i(\mu_{i}(s)+1)}\ge s+1\ge T_{(i+1)\mu_{i+1}(s+1)}$, we have $\mu_{i}(s)+1>\mu_{i+1}(s+1)$, i.e.,
\[
\mu_{i}(s)\ge \mu_{i+1}(s+1).
\]
Since
\begin{eqnarray*}
\{j\mid Q_{ij}\le m\}&=&\bigsqcup_{k=1}^m\{j\mid Q_{ij}=m-k+1\}\\
&=&\bigsqcup_{k=1}^m\{j\mid \mu_{k+i}(n+k)<j\le\mu_{k+i-1}(n+k-1)\}\\
&=&\{j\mid 1\le j\le\lambda_i\},
\end{eqnarray*}
it follows that $Q$ has shape $\lambda$. To prove $\type(Q)=\beta$, since $\mu_i(s)$ is the maximal $j$ such that $T_{ij}\le s$, it follows that $T_{ij}=s$ iff $\mu_{i}(s-1)<j\le \mu_i(s)$. Thus for $k\le i\le m$,
\[
\{j\mid T_{ij}=n+k\}=\{j\mid\mu_{i}(n+k-1)< j\le \mu_{i}(n+k)\}.
\]
By the construction of $Q_{ij}$,
\[
\{j\mid Q_{(i-k)j}=m-k+1\}=\{j\mid \mu_{i}(n+k)<j\le \mu_{i-1}(n+k-1)\}.
\]
For $k\le i\le m+1,$ we have
\[
\#\{j\mid T_{ij}=n+k\}+\#\{j\mid Q_{(i-k)j}=m-k+1\}=\mu_{i-1}(n+k-1)-\mu_i(n+k-1).
\]
Note that $\mu_{m+1}(k)=0$, $\mu_i(n+k)=d$ for $i\le k$. Then $\sum_{i=k}^{m+1}(\mu_{i-1}(n+k-1)-\mu_i(n+k-1))=d$. For $i<k$, there is no $j\in [d]$ such that $Q_{(i-k)j}=m-k+1$ and $T_{ij}=n+k$. Similar as before, we have
\[
\#\{(i,j)\mid T_{ij}=n+k\}+\#\{(i,j)\mid Q_{ij}=m-k+1\}=d.
\]
This proves that $Q$ is of type $\beta$. It remains to show that $Q$ is a SSYT.
Note that
\begin{eqnarray*}
\{j\mid Q_{ij}=m-k+1\}&=&\{j\mid \mu_{k+i}(n+k)<j\le \mu_{k+i-1}(n+k-1)\},\\
\{j\mid Q_{ij}=m-k\}&=&\{j\mid \mu_{k+i+1}(n+k+1)<j\le \mu_{k+i}(n+k)\}.
\end{eqnarray*}
It follows that in the $i$-th row of $Q$, the boxes filled with the number $m-k$ is strictly left of the boxes filled with the number $m-k+1$, which means $Q$ is weakly increasing across each row. Suppose there exist $i$ and $j$ such that $Q_{ij}=m-k+1$ and $Q_{(i+1)j}=m-k'+1$ with $k<k'$. Then we have
\[
\mu_{k'+i+1}(n+k')<j\le \mu_{k'+i}(n+k'-1)\And \mu_{k+i}(n+k)<j\le \mu_{k+i-1}(n+k-1),
\]
in contradiction with
\[
\mu_{k'+i}(n+k'-1)\le \mu_{k'+i}(n+k')\le \mu_{k+i}(n+k).
\]
Therefore $Q$ is a SSYT and the map $\theta$ is a bijection.
\end{proof}
Indeed, we can obtain $(P,Q)$ from $T$ by inverting Algorithm \ref{algorithm}.
\begin{example}
Let $T$ be a SSYT of shape $( 5^{5})$ and type $(\alpha,5-\beta)=(2,2,3,2,3,4,4,3,2)$ for $\alpha=(2,2,3,2)$ and $\beta=(2,1,1,2,3)$. We construct a pair $(P,Q)$ of SSYTs of the same shape from $T$  as follows,
\[
\begin{array}{c}T=\\{}\\{}\\{}\\{}\\\end{array}
\young(11235,23466,34577,56788,67899)
\begin{array}{c} \longrightarrow \\{}\\{}\\{}\\\end{array}
\begin{array}{c}P=\\{}\\{}\\{}\\{}\\\end{array}
\young(1123,234,34)
\begin{array}{c} \longrightarrow \\{}\\{}\\{}\\\end{array}
\young({}{}{}{}{},{}{}{}{},{}{}{})
\vspace{-10mm}
\]
\[
\begin{array}{c} \rightarrow \\{}\\{}\\\end{array}
\young(111{}{},{}{}{}{},{}{}{})
\hspace{0mm}
\begin{array}{c} \rightarrow \\{}\\{}\\\end{array}
\young(111{}{},22{}{},{}{}{})
\hspace{0mm}
\begin{array}{c} \rightarrow \\{}\\{}\\\end{array}
\young(111{}{},22{}{},3{}{})
\hspace{0mm}
\begin{array}{c} \rightarrow \\{}\\{}\\\end{array}
\young(111{}{},224{},3{}{})
\hspace{0mm}
\begin{array}{c} \rightarrow \\{}\\{}\\\end{array}
\begin{array}{c}Q=\\{}\\{}\\{}\\{}\\\end{array}
\young(1115,224,35).
\]
\end{example}
Note that if $\type(P)=(\alpha_1,\ldots,\alpha_n)$ and $\type(Q)=(\beta_1,\ldots,\beta_m)$, $\sh(P)=\sh(Q)\subseteq (d^m)$ iff $\sh(P)=\sh(Q)\subseteq (d^n)$. By the definition of $\mathfrak{S}_{d}(\alpha,\beta)$, we have
\[
\left|\mathfrak{S}_{d}(\beta,\alpha)\right|=\left|\mathfrak{S}_{d}(\alpha,\beta)\right|=\sum_{\lambda\vdash N;\lambda\subseteq (d^m)}K_{\lambda\alpha}K_{\lambda\beta}.
\]
By Theorem \ref{bijection on SSYT}, we have
\begin{eqnarray*}
&&\left|\mathfrak{S}_{d}(\alpha,\beta)\right|= \left|\mathfrak{T}_{d}(\alpha,\beta)\right|=K_{(d^m)(\alpha,d-\bar{\beta})}
=K_{(d^m)(\alpha,d-\beta)},\\
&&\left|\mathfrak{S}_{d}(\beta,\alpha)\right|= \left|\mathfrak{T}_{d}(\beta,\alpha)\right|=K_{(d^n)(\beta,d-\bar{\alpha})}
=K_{(d^n)(\beta,d-\alpha)}.
\end{eqnarray*}

\begin{corollary}Let  $\alpha=(\alpha_{1},\ldots,\alpha_{n})$ and $\beta=(\beta_{1},\ldots,\beta_{m})$ be two compositions of $N$. Suppose $m\le n$. We have
\[
\sum_{\lambda\vdash N;\lambda\subseteq (d^m)}K_{\lambda\alpha}K_{\lambda\beta}=K_{(d^m)(\alpha,d-\beta)}=K_{(d^n)(\beta,d-\alpha)}.
\]
\end{corollary}

Recall that RSK correspondence is a bijection sending each generalized permutation of $\alpha\to \beta$ to a pair of SSYTs of the same shape and types $\alpha$ and $\beta$ respectively, i.e., if $d=N$,
\[
|S_{\alpha\beta}|=|\mathfrak{S}_{N}(\alpha,\beta)|=\sum_{\lambda\vdash N}K_{\lambda\alpha}K_{\lambda\beta}=K_{( N^{m})(\alpha,N-\beta)}=K_{(N^n)(\beta,N-\alpha)},
\]
which proves Theorem \ref{Kostka theorem}.

\section{Applications}
In this section, we will give some applications of previous results. Let $S_{\alpha\beta}^k$ be the set of generalized permutations $\tau\in S_{\alpha\beta}$ whose second row $\tau_{11}\tau_{12}\ldots$ has no weakly increasing subsequence of length $k$, i.e., no index sets $1\leq i_{1}\le i_{2}\le\cdots\le i_{k}$ and $j_{1},j_{2},\ldots ,j_{k}$ satisfy $\tau_{i_{1}j_{1}}\le\tau_{i_{2}j_{2}}\le\cdots\le\tau_{i_{k}j_{k}}$. When $\beta=(1^{\sum\alpha_i})$, each member of $S_{\alpha\beta}^{k}$ is a block-ascending permutation which was first studied in \cite{Lewis2011}.
\begin{theorem}\label{L2011}{\rm\cite{Lewis2011}} If $N=kn$, there are two bijections between
\begin{itemize}
  \item $ S_{(k^{n}) ( 1^{N})}^{k+1}$ and the set of SYTs of shape $( k^{n})$,
  \item $ S_{( (k-1)^{n}) ( 1^{N-n})}^{k+1}$ and the set of SYTs of shape $( k^{n})$.
\end{itemize}
By the hook length formula, we have
\[
|S_{(k^{n}) ( 1^{N})}^{k+1}|=|S_{( (k-1)^{n}) ( 1^{N-n})}^{k+1}|=f^{(k^n)}.
\]
\end{theorem}
His proof is a revisional RSK correspondence which merges a pair $(P,Q)$ of SYTs obtained by RSK correspondence from a block-ascending permutation to a SYT of shape $( k^{n})$. Analogous to Lewis's construction, we \cite{Mei} obtained the following extension.
\begin{theorem}\label{M2017}
{\rm \cite{Mei}}
Let $\alpha=(\alpha_{1},\ldots,\alpha_{n})$ be a composition of $N$ and $\alpha_i\in \{k-1,k\}$ for all $i\in [n]$. There is a bijection between $ S^{k+1}_{\alpha(1^{N})}$ and the set of SYTs of shape $ (k^{n})$. Then by the hook length formula, the cardinality
\[
|S^{k+1}_{\alpha(1^{N})}|=f^{(k^{n})}
\]
is independent of the choice of $\alpha_i\in \{k-1,k\}$ for all $i\in [n]$.
\end{theorem}
Later, Chen \cite{Chen} proved that for any composition $\alpha$ of $N$, the cardinality of $S^{k+1}_{\alpha(1^{N})}$ is independent of the order of entries of $\alpha$. As byproducts, he gave a direct bijection between $S^{k+1}_{\alpha(1^{N})}$ and $S^{k+1}_{(k^n)(1^{N})}$  for any $\alpha\in \{k-1,k\}^n$ without RSK correspondence involved,  and proved that the cardinality of $S^{k+1}_{\alpha( 1^{N})}$ is Schur-concave, i.e., if $\alpha=(\alpha_{1},\ldots,\alpha_{n})$ and $\beta=(\beta_{1},\ldots,\beta_{m})$ are partitions of $N$ with $\alpha_{1}+\cdots+\alpha_{i}\geq \beta_{1}+\cdots+\beta_{i}$ for all $i\geq 1$ under the assumption of $\alpha_i=\beta_j=0$ for $i>n$ and $j>m$,
\[
| S^{k+1}_{\alpha( 1^{N})}|\leq |S^{k+1}_{\beta( 1^{N})}|.
\]

If $\tau\in S^{k+1}_{\alpha\beta}$ and $\tau\xrightarrow{RSK}(P,Q)$, it is well known that the first row of $\sh(P)=\sh(Q)$ is no longer than $k$. Indeed, RSK correspondence gives a bijection between $S^{k+1}_{\alpha\beta}$ and $\mathfrak{S}_{k}(\beta,\alpha)$. The following result recovers Chen's Schur-concavity \cite{Chen} of $|S^{k+1}_{\alpha(1^{N})}|$ on the partition $\alpha$ and Theorem \ref{M2017} by taking $\alpha\in \{k-1,k\}^n$ and $\beta=(1^N)$.
\begin{corollary}\label{cor-6}
Let  $\alpha$ and $\beta$ be two compositions of $N$. If $k\in \Bbb{N}$, then
\[
|S^{k+1}_{\alpha\beta}|=|S^{k+1}_{\beta\alpha}|=\sum_{\lambda\vdash N;\lambda\subseteq (k^m)}K_{\lambda\alpha}K_{\lambda\beta}
=K_{(k^m)(\alpha,k-\beta)}
\]
is independent of the order of entries of $\alpha$ and $\beta$. Moreover, $|S^{k+1}_{\alpha\beta}|$ is Schur-concave on partitions $\alpha$ and $\beta$. In particular,
\[
|S^{k+1}_{\alpha\beta}|\le |S_N(12\cdots (k+1))|,
\]
where $S_N(12\cdots (k+1))$ is the set of all $12\cdots (k+1)$-avoiding permutations on $[N]$.
\end{corollary}
From Corollary \ref{cor-6}, taking $\alpha=1^N$ and $k=2$, we have
\[
|S^3_{( 1^{N})\beta}|=K_{(2^m)(1^N,2-\beta)}.
\]
\begin{corollary}
Let $\alpha=(\alpha_1,\ldots,\alpha_n)$ be a composition of $N$ with $\alpha_i>0$. If $\alpha\in \{1,2\}^n$, then the number of permutations on the multiset $[\alpha]$ which contain no weakly increasing subsequence of length $3$ is the $n$-th Catalan number, i.e.,
\begin{eqnarray*}
|S_{( 1^{N})\alpha}^3|=\left\{\begin{array}{lcl}C_{n}\quad\quad~\text{if}~~ \alpha\in \{1,2\}^n,\\ 0\quad\quad\quad\text{otherwise}.\end{array}\right.
\end{eqnarray*}
\end{corollary}
Let $P$ and $Q$ be two SSYTs of the same shape $( n^{2})$ and types $( 2^{n})$ and $(1^{2n})$ respectively. If $(P,Q)\xrightarrow{RSK} \tau$, it follows from RSK correspondence \cite{Stanley-Vol2} that
\[
\tau\in S_{( 1^{2n})( 2^{n})}(321)\cap S^{n+1}_{( 1^{2n})( 2^{n})}.
\]
 From Theorem \ref{generalized Riordan}, the number of such pairs $(P,Q)$ is $C_{n}\times R_{n}$.
\begin{corollary}\label{Catalan-Riordan}The number of permutations on the multiset $\{1,1,2,2,\ldots,n,n\}$ which avoid the pattern $321$ and contain no weakly increasing subsequence of length $n$ is $C_nR_n$, i.e.,
\[
C_{n}R_{n}=|S_{( 1^{2n})( 2^{n})}(321)\cap S^{n+1}_{( 1^{2n})( 2^{n})}|.
\]
\end{corollary}

If $\alpha=(\alpha_1,\ldots,\alpha_n)$ is a composition of $N$, it is known that the number of permutations on the multiset $[\alpha]$ is ${N\choose \alpha_1,\ldots,\alpha_n}$,
i.e.,
\[
|S_{(1^{N})\alpha}|={N\choose \alpha_1,\ldots,\alpha_n}.
\]
\begin{corollary}If $\alpha$ is a composition of $N$ and $f^\lambda$ the number of SYTs of shape $\lambda$, then
\[
\sum_{\lambda\vdash N}f^\lambda K_{\lambda\alpha}={N\choose \alpha_1,\ldots,\alpha_n}.
\]
\end{corollary}
From the theory of symmetric functions \cite{Macdonald,Stanley-Vol2},
\[
h_{\lambda}=\sum_{\mu}(\sum_{\nu}K_{\nu\lambda}K_{\nu\mu})m_{\mu},
\]
where $h_{\lambda}$ is the complete symmetric function and $m_{\mu}$ the monomial symmetric function. Theorem \ref{Kostka theorem} implies the following result.
\begin{corollary}
If $\lambda\vdash N$ and  the length of $\lambda$ is $\ell(\lambda)$, then
\[
h_{\lambda}=\sum_{\mu\vdash N}K_{( N^{\ell(\lambda)})(\mu,N-\lambda)}m_{\mu},
\]
i.e., the transition matrix from the basis $(h_{\lambda})$ to the basis $(m_{\lambda})$ consists of Kostka numbers of rectangular shape.
\end{corollary}


\begin{thebibliography}{10}
\bibitem{Albert} M. H. Albert, R. E. L. Aldred, M. D. Atkinson, C. Handley, and D. Holton, { Permutations of a Multiset Avoiding Permutations of Length 3}, European J. Combin. 22 (2001), no. 8, 1021-1031.
\bibitem{Atkinson} M. D. Atkinson, S. A. Linton, and L. A. Walker, { Priority Queues and Multisets}, Electron. J. Combin. 2 (1995), \#P24.
\bibitem{Baril2011}J.-L. Baril, Classical sequences revisited with permutations avoiding dotted pattern, Electronic Journal of Combinatorics, 18 (2011), \#P178.
\bibitem{Bernhart} F. R. Bernhart, { Catalan, Motzkin, and Riordan Numbers}, Discrete Math. 204 (1999), no. 1-3,  73-112.
\bibitem{Burstein2014}A. Burstein, J. Pantone, Two examples of unbalanced Wilf-equivalence, arXiv:1402.3842 [math.CO], 2014.
\bibitem{Chen YC} William Y. C. Chen, Eva Y. P. Deng, and Laura L. M. Yang, { Riordan Paths and Degrangements}, Discrete Math. 308(2008), no. 11, 2222-2227.
\bibitem{Chen} E. Chen, { Schur-Concavity for Avoidance of Increasing Subsequences in Block-Ascending Permutations}, Electron. J. Combin. 24 (2017), no. 4, \#P4.4.
\bibitem{Knuth1973} D. E. Knuth, { The Art of Computer Programming. Vol. III}, Addison-Wesley Reading MA, 1973.
\bibitem{Lewis2011} J. B. Lewis, {Pattern Avoidance for Alternating Permutations and Young Tableaux}, J. Combin. Theory Ser. A 118 (2011), no. 4, 1436-1450.
\bibitem{Macdonald} I. G. Macdonald, { Symmetric Functions and Hall Polynomials}, 2nd Edition, Oxford University Press, Oxford, 1995.
\bibitem{MacMahon} P. A. MacMahon, { Combinatory Analysis}, London Cambridge University Press, Volume I, Section III, Chater 1915.
\bibitem{Mei} Z. S. Mei, S. J. Wang, {  Pattern Avoidance and Young Tableaux}, Electron. J. Combin. 24 (2017), no. 1, \#P1.6.
\bibitem{Myers} A. N. Myers, { Pattern Avoidance in Multiset Permutations: Bijective Proof}, Ann. Comb. 11 (2007), no. 3-4, 507-517.
\bibitem{Regev2010} A. Regev, { Identities for the Number of Standard Young Tableaux in Some $(k,\ell)$-Hooks}, S\'{e}minaire Lotharingien de Combinatoire, 63(3), 2010.
\bibitem{Savage} C. D. Savage and H. S. Wilf, { Pattern Avoidance in Compositions and Multiset Permutations}, Adv. in Appl. Math. 36 (2006), no. 2, 194-201.
\bibitem{Simion} R. Simion and F. W. Schmidt, { Restricted Permutations}, European J. Combin. 6 (1985), no. 4, 383-406.
\bibitem{Sloane} N. J. A. Sloane, The On-Line Encyclopedia of Integer Sequence, Published electronically at http://www.research.att.com/$\sim$njas/sequences/.
\bibitem{Stanley-Vol1} R. P. Stanley, { Enumerative Combinatoric}, vol.1, Cambridge University Press, 2001.
\bibitem{Stanley-Vol2} R. P. Stanley, { Enumerative Combinatoric}, vol.2, Cambridge University Press, 2001.
\bibitem{West1990} J. West, { Permutations with Forbidden Subsequences and Stack-Sortable Permutations}, Ph.D. Thesis, M.I.T., Cambridge, MA, 1990.
\bibitem{White} D. E. White, { Monotonicity and Unimodality of the Pattern Inventory}, Adv. in Math. 38 (1980), no. 1, 101-108.
\end{thebibliography}
\end{document}